\documentclass[twoside, 12pt]{amsart}
\usepackage{latexsym}
\usepackage{amssymb,amsmath,amsopn}
\usepackage[dvips]{graphicx}   
\usepackage{color,epsfig}      

\newtheorem{thm}{Theorem}
\newtheorem{lem}{Lemma}

\newtheorem{defn}{Definition}
\newtheorem{rem}{Remark}
\newtheorem{prob}{Problem}

\newtheorem{conj}{Conjecture}
\newtheorem{ques}{Question}

\DeclareMathOperator{\conv}{conv}

\DeclareMathOperator{\dif}{d}
\DeclareMathOperator{\vol}{vol}
\DeclareMathOperator{\surf}{surf}
\DeclareMathOperator{\area}{area}
\DeclareMathOperator{\inter}{int}
\DeclareMathOperator{\aff}{aff}
\DeclareMathOperator{\relint}{relint}

\DeclareMathOperator{\bd}{bd}

\DeclareMathOperator{\perim}{perim}
\newcommand{\K}{\mathcal{K}}

\renewcommand{\Re}{\mathbb R}
\newcommand{\Sph}{\mathbb{S}}

\newcommand{\M}{\mathcal{M}}
\newcommand{\B}{\mathbf B}
\newcommand{\F}{\mathcal{F}}
\newcommand{\C}{\mathcal{C}}
\newcommand{\W}{\mathcal{W}}

\begin{document}
\title[On the average number of normals]{On the average number of normals through points of a convex body}
\author[G. Domokos and Z. L\'angi]{G\'abor Domokos and Zsolt L\'angi}
\address{G\'abor Domokos, Dept. of Mechanics, Materials and Structures, Budapest University of Technology,
M\H uegyetem rakpart 1-3., Budapest, Hungary, 1111}
\email{domokos@iit.bme.hu}
\address{Zsolt L\'angi, Dept.\ of Geometry, Budapest University of
Technology and Economics, Budapest, Egry J\'ozsef u. 1., Hungary, 1111}
\email{zlangi@math.bme.hu}

\thanks{The authors gratefully acknowledge the support of the J\'anos Bolyai Research Scholarship of the Hungarian Academy of Sciences and support from OTKA grant 104601}

\begin{abstract}
In 1944, Santal\'o asked about the average number of normals through a point of a given convex body. Since then, numerous results
appeared in the literature about this problem.
The aim of this paper is to give a concise summary of these results, with some new, recent developments.
We point out connections of this problem to static equilibria of rigid bodies as well as to geometric partial differential equations of surface evolution.
\end{abstract}

\maketitle

\section{Introduction}\label{sec:intro}

\subsection{Motivation: static equilibria}

The study of equilibrium points of a convex solid $K$, 
with respect to its centre of gravity $g$ has been a fundamental question of statics 
ever since the work of Archimedes \cite{Archimedes}.
The number $n(K,g)$ of equilibria (i.e. the number of surface normals passing through $g$) is characteristic of the shape, it has been applied to classify beach pebbles \cite{DSSV} and turtle shells \cite{DV}.

It is a natural question to ask how difficult it is to change $n(K,g)$;
we call this property \emph{robustness}, and in \cite{DL1} we introduced possible mathematical approaches to this concept.
The question is strongly motivated by the above-mentioned applications in natural sciences: material inhomogeneities can be interpreted as a variation of 
the location of $g$, while surface evolution (e.g. by abrasion) results in a variation of the hull $\bd K$.
The main difficulty of the mathematical problem lies in the nontrivial coupling via integrals between $g$ and $\bd K.$

One possible approach (which we followed in \cite{DL1}) is to solve easier, decoupled problems. We can generalize the concept $n(K,g)$ of static equilibria
(with respect to the centre of gravity $g$) to that of \emph{relative equilibria}  $N(K,p),$ with respect to a fixed point $p \in K$. Motivated by material inhomogeneities,
we may seek \emph{internal robustness}, defined as the radius $r$ of the maximal ball $\B_r$ around $g$ such that if $p\in \B_r$ then $n(K,p)=n(K,g)$.
From the mathematical point of view, this problem is closely related to the geometry of caustics \cite{PS96},
and has attracted recent interest in the context of inhomogeneous polyhedra (cf. \cite{Dawson}, \cite{Dawson2}, \cite{DawsonFinbow} and \cite{Heppes}).
The other decoupled problem, motivated by the discrete model of abrasion and fragmentation \cite{DSV}, \cite{Krapivsky}, leads to the concept of \emph{external robustness}, defined as the minimal (scaled) truncation of $K$,
resulting in the convex body $K'$ such that $n(K',g)<n(K,g)$.

In the current paper we follow another approach: by averaging $n(K,p)$ over all possible locations $p\in K$ we
obtain the quantity $n(K)$ to the study of which our paper is dedicated. Obviously, $n(K)$ will be independent of internal perturbations (material inhomogeneities) and
one would expect it to be less sensitive to external perturbations.
As we will discuss it in Section \ref{sec:remarks}, this expectation is partially justified:
$n(K)$ changes continuously if we consider local truncations of $K$ by hyperplanes, corresponding to the above-mentioned discrete models of abrasion processes \cite{DSV}.
It is remarkable that these mathematical questions have been studied in great detail in their own right, regardless of the aforementioned physical
motivation. In 1944, Santal\'o \cite{S44} asked about the extremal values of the average number of normals through a point of a convex body $K$; or in other words, the extremal values of $n(K)$.
Since then, numerous results appeared in the literature regarding this quantity and
one main aim of our paper is to review them and to present some new results. 
Among others, we improve several existing bounds for $n(K)$ (corresponding to various classes of convex bodies)
and we also sharpen the necessary conditions for centrally symmetric polytopes corresponding to maximal values of $n(K)$.

Beyond adding to existing results on extremal values of $n(K)$, 
our second goal is to point out a connection to geometric PDEs.
In 1987 Grayson \cite{Grayson} proved that in two dimensions, under the so-called \emph{curve shortening flow} (where points of $\bd K$ move
in the direction of the inward surface normal with speed proportional to the curvature) $N(K,p)$ is decreasing monotonically.
We will point out that, under some restrictions on $K$,  $n(K)$ is increasing monotonically under the \emph{uniform flow} (where points
of $\bd K$ move in the direction of the inward surface normal with uniform speed) and we will state a conjecture that similar
behaviour is expected if the normal speed is proportional to some negative power of the curvature. Since our result is not
restricted to two dimensions and $n(K)$, unlike $n(K,p)$, is a frame-invariant quantity, it opens a potentially interesting
connection to the theory of geometric PDEs. There has been substantial
interest in the monotonicity of frame-invariant quantities under geometric PDEs, for example the Huisken functional
under the \emph{mean curvature flow} \cite{Huisken}, the isoperimetric ratio (in two dimensions) under the curve shortening flow
\cite{Gage} and the curvature entropy under the same flow \cite{Chow}.
Beyond being mathematically challenging, geometric PDEs are also
interesting from the point of view of natural abrasion processes.  The shape and evolution of pebbles has been a matter of discussion
since at least the time of Aristotle \cite{Krynine1} and even in recent times it has received the attention of mathematicians \cite{Bloore} \cite{Firey}, \cite{Rayleigh1}.
However, caution is advised when using $n(K)$ as an indicator in these PDE models. As we discuss it in Section \ref{sec:remarks}, while $n(K)$ changes continuously
under $C^2$-smooth variation of $\bd K$, 
it is easy to show an example where $n(K)$ changes discontinuously under a $C^0$-smooth change of $\bd K$.

After introducing basic notions and notations in Subsection~\ref{ssec:notations},
we  review existing results on $n(K)$ in Section \ref{sec:history}. Our new results,
including those on evolution PDEs, are presented in Section \ref{sec:newresults}.
In Section \ref{sec:remarks} we state some open problems, conjectures and
make some additional comments.


\subsection{Definitions and basic notations}\label{ssec:notations}

In our investigation, we denote Euclidean $m$-space space by $\Re^m$, and its unit sphere, with the origin $o$ as its centre, by $\Sph^{m-1}$.
For $p,q \in \Re^m$, $[p,q]$ denotes the closed segment with endpoints $p$ and $q$.
We denote $m$-dimensional Lebesgue measure by $\vol$, or, in the case $m=2$, by $\area$.
A \emph{convex body} is a compact, convex set with nonempty interior. A convex body $K$ is of \emph{constant width $s$}, if the distance between any two parallel supporting hyperplanes of $K$ is $s$. The surface area of a convex body $K \subset \Re^m$ is denoted by $\surf(K)$, or in the case $m=2$, by $\perim(K)$.

Our main definition is the following.

\begin{defn}\label{defn:equilibrium}
Let $K \subset \Re^m$ be a convex body, and let $q \in \bd K$. A line, starting at $q$ and perpendicular to a supporting hyperplane of $K$ at $q$
is called a \emph{normal} of $K$ at $q$. For any $p \in K$, we denote the number of the normals of $K$, passing through $p$, by $n(K,p)$.
Furthermore, we set
\[
I(K) = \int_{p \in K} n(K,p) \dif p, \quad \mathrm{and} \quad n(K) = \frac{I(K)}{\vol(K)} .
\]
\end{defn}

Observe that the quantity $n(K)$ is the mean value of the number of normals through a randomly chosen point of $K$, using uniform distribution.

\section{History}\label{sec:history}

The first result regarding Santal\'o's question is due to Chakerian \cite{C84}, who examined a similar problem for plane convex bodies.

\begin{defn}\label{defn:diameter}
Let $K \subset \Re^m$ be a convex body. If $[p,q] \subset K$ with the property that for any $[r,s] \subset K$ parallel to $[p,q]$, $|q-p| \geq |s-r|$
is satisfied, then we say that $[p,q]$ is an \emph{affine diameter} (or shortly, \emph{diameter}) of $K$.
For any $p \in K$, we denote the number of affine diameters of $K$, containing $p$, by $d(K,p)$.
Furthermore, we set
\[
D(K) = \int_{p \in K} d(K,p) \dif p, \quad \textrm{and} \quad d(K)=\frac{D(K)}{\vol(K)}.
\]
\end{defn}

Chakerian observed that if $K$ is a convex body of constant width, the diameters of $K$ are exactly the normals of $K$, counted once belonging to each of its endpoints; or in other words, that for any $p \in K$, we have $d(K,p) = 2 n(K,p)$. In his paper he showed that
\[
\frac{1}{4} \area(K-K) \leq D(K) \leq \frac{1}{2} \area(K-K)
\]
for any plane convex body $K$ with $C^3$-differentiable boundary, and with nowhere vanishing curvature.
Combining this with the estimates on the volume of the difference body of a convex body (cf. e.g. \cite{BF34}), he obtained that
$\area(K) \leq D(K) \leq 3 \area(K)$.
Here equality on the left-hand side is attained only by centrally symmetric bodies, and the quantity on the right can be approached, for example, by slightly modified Reuleaux polygons.

By the Blaschke-Lebesgue theorem (cf. \cite{BF34}), for plane convex bodies of constant width it follows that
\begin{equation}\label{eq:BlaschkeLebesgue}
2 \leq n(K) \leq \frac{2\pi}{\pi-\sqrt{3}},
\end{equation}
with equality on the left if $K$ is a circle, whereas the right-hand side can be approached by slightly modified Reuleaux triangles.
Here, the right-hand side inequality was proven also in \cite{S44}.

The method of Chakerian was modified and generalized by Hann \cite{Hann93} for any convex body $K \subset \Re^m$, under the assumptions that $K$ is either a polytope, or a strictly convex body with $C^2$-class boundary.
More specifically, she proved that if $K$ is any such plane convex body, then
\begin{itemize}
\item $n(K) \leq 12$,
\item $n(K) \leq 8$ if $K$ is centrally symmetric, and
\item $n(K) \leq 6$, if all the centres of curvature of $K$ are contained in $K$.
\end{itemize}
We note that her last estimate is a generalization of the estimate $D(K) \leq 3 \area(K)$ in \cite{C84}.

Furthermore, if $K \subset \Re^m$, then
\begin{equation}\label{eq:gen_estimate}
n(K) \leq \frac{\vol(2K-K)}{\vol(K)}-1.
\end{equation}
If $K$ is centrally symmetric, the bound in (\ref{eq:gen_estimate}) gives $n(K) \leq 3^m-1$, which is attained, for example, for cubes (cf. \cite{Hann93} or \cite{Hug95}).
If $K$ is not symmetric, we may use estimates on the volume of the difference body \cite{RS57}, and the equality of the mixed volumes
$V(K,K,-K)$ and $V(K,-K,-K)$ to obtain numeric upper bounds for $n(K)$ in dimensions $2$ and $3$ (cf. \cite{Hann96}).
Table \ref{tab1} shows the known upper bounds in these dimensions on $n(K)$ over the examined families: those of $m$-dimensional convex bodies, $o$-symmetric convex bodies, convex bodies containing all their centres of curvature, and convex bodies of constant width, denoted by $\K_m$, $\M_m$, $\C_m$ and $\W_m$, respectively. Note that as the last three families are subfamilies of $\K_m$, the upper bound for $\K_m$ trivially holds in any of the four classes.
Here we used the result of Rogers and Shephard \cite{RS57} to estimate the ratio of $\vol(2K-K)$ to $\vol(K)$.

\begin{center}
\begin{table} \label{tab1}
\begin{tabular}{|l||c|c|c|c|}
\hline
 & $\K_m$ & $\M_m$ & $\C_m$ & $\W_m$ \\
\hline
\hline
$m=2$ & $12$ & $8$ & $6$ & $\frac{2\pi}{\pi-\sqrt{3}}$ \\
\hline
$m=3$ & $62$ & $26$ & ? & ? \\
\hline
$m$ & $\left( \frac{3}{2} \right)^m \binom{2m}{m} - 1$ & $3^m -1$ & ? & ? \\
\hline
\end{tabular}
\caption{Upper bounds of $n(K).$ Rows correspond to dimensions ($m$), columns correspond to   convex bodies ($\K_m$), $o$-symmetric convex bodies ($\M_m$), convex bodies containing all their centres of curvature ($\C_m$), and convex bodies of constant width ($\W_m$). Question marks denote unknown upper bounds, to be computed in the current paper.}
\end{table}
\end{center}

The results of Hann were generalized by Hug \cite{Hug95}, who applied a new type of approximation process to obtain the same bounds without the restriction that $K$ be a polytope or sufficiently smooth.
We remark that it is not known if there is any convex body satisfying $n(K) > 3^m-1$, even in the planar case,
and thus, in this regard, the problem of finding the supremum of the quantity $n(K)$ over $\K_m$ is still open.
An interesting result appeared in \cite{D98}, showing that for any convex polygon $P$ in $\Re^2$, we have $4 < n(P)$, where $4$ can be approached by a suitable sequence of polygons. This suggests that $n(K)$ is ``larger'' for polytopes than for smooth convex bodies.

Hann \cite{Hann96} proposed a normed version of the problem. In this, to define normals, she used the so-called Birkhoff orthogonality relation \cite{MS04}.
More specifically, in a Minkowski space with unit ball $M$, we say that a line $L$ is \emph{normal} to a hyperplane $H$, if they have some translates
$L'$ and $H'$, respectively, such that $H'$ is a supporting hyperplane of $M$, $L'$ contains the origin $o$, and $M' \cap L'$ is contained in $\bd M$.
Furthermore, if $K$ is a convex body, then a line $L$ is a \emph{Minkowski normal} (or shortly, \emph{normal}) of $K$, if $L$ is normal to
a hyperplane $H$, supporting $K$ at some point of $L$.
Then, one may define $n_M(K,p)$ as the number of Minkowski normals through $p$ in the Minkowski space with unit ball $M$,
set $I_M(K) = \int_K n_M(K,p) \dif p$, and estimate the ratio $n_M(K) = \frac{I_M(K)}{\vol(K)}$.
Observe that since any normed volume is only a scalar multiple of the standard Lebesgue measure, it does not matter what volume we use
in the definition above.

For this version, Hann \cite{Hann97} proved that if $M$ is a smooth and strictly convex body in the plane, and $K$ is either a polygon or has $C^2$-class boundary,
then
\[
n_M(K) \leq 12,
\]
where $12$ can be replaced by $8$ if $K$ is centrally symmetric, and by $6$ if all the centres of circular curvature of $K$ are inside $K$.
Furthermore, in \cite{Hann99}, she proved the first two estimates for the case that $M$ is a polygon, and $K$ is either a polygon with no side parallel to a side of $M$, or if $K$ has $C^2$-class boundary.

As the last result, we mention an Euler-type formula in \cite{Hann93}, which, for the case $m=3$, has been proved by elementary methods in \cite{CH99}.

\section{New results}\label{sec:newresults}

We present our results in three subsections. In the first one, we try to find the minimizer and maximizers of $n(K)$ in certain families of convex bodies. In the second one, we collected those involving bodies that contain all their centres of curvature;
in particular bodies of constant width.
In the third one we deal with the planar case of the problems in Section~\ref{sec:history}.

\subsection{Minimizers/maximizers of $n(K)$}

A natural problem (cf. \cite{Hann96}) is to try to characterize the convex bodies $K$ for which $n(K)$ is minimal or maximal.
Clearly, as a lower bound, one can give the trivial estimate $n(K) \geq 2$ for any convex body.
This observation was made, for example, in \cite{Hann93}.
We prove the following, stronger version of this observation, which we use in Theorem~\ref{thm:eikonal}.

\begin{thm}\label{thm:lowerbound}
Let $K \in \in \Re^m$ be a convex body, with $C^2$-differentiable boundary. If $n(K) = 2$, then $K$ is a Euclidean ball.
Furthermore, if $K$ is a polytope, then $n(K) > 2$.
\end{thm}

\begin{proof}
First, assume that $K$ has $C^2$-class boundary, and it is not a Euclidean ball. Let $\kappa$ be the maximum of the principal curvatures of $\bd K$. Let $p \in \bd K$ be a point where this maximum is attained, and let $q$ be the corresponding centre of curvature of $K$.
By Blaschke's Rolling Ball Theorem \cite{Blaschke},  the sphere $q + \frac{1}{\kappa} \Sph^{n-1}$ is contained in $K$.
Then the Euclidean distance function $x \mapsto |x-q|$, where $x \in \bd K$, attains its absolute minimum at $p$.

Let $r \notin [p,q]$ be a point of the normal of $K$ at $p$, sufficiently close to $q$.
Clearly, $p$ is a critical point, but not a local (and thus an absolute) minimum of the function $x \mapsto |x-r|$, where $x \in \bd K$.
On the other hand, this function attains its minimum and its maximum at some points, which, since $K$ is not a Euclidean ball, are different from $p$,
which yields that $n(K,q) \geq 3$. By the $C^2$-differentiability of the Euclidean distance function, the same holds in a neighborhood of $q$, and thus,
$n(K) > 2$.

To prove the assertion for the case that $K$ is a polytope, we may apply a similar argument, replacing the inner normal at $p$ by a chord connecting
two farthest vertices of $K$.
\end{proof}

To find the maximizers of $n(K)$ seems much more complicated, especially since, numerically, not even the maximal value is known in $\K_m$.
This question seems more approachable if the maximum is taken over $K \in \M_m$, since for this family, the maximal value is known to be $3^m -1$.
Regarding this problem, Hug \cite{Hug95} showed that if $P$ is a centrally symmetric convex polygon, then $n(P) = 8$ if, and only if,
$P$ is inscribed in a circle.
In \cite{Hann96}, Hann remarked that this method can be generalized to any dimension, and yields that if $P$ is a centrally symmetric
convex polytope, then $n(P) = 3^m -1$ if, and only if, $P$ is \emph{blocklike}; that is, a centrally symmetric polytope inscribed in a sphere.
Nevertheless, as Theorem~\ref{thm:equality} shows, this remark is incorrect.
Before formulating it, let us recall that a zonotope is the Minkowski sum of finitely many closed segments; or equivalently, a centrally symmetric polytope with centrally symmetric faces \cite{McMullen}.

\begin{thm}\label{thm:equality}
Let $P \in \M_m$ be a convex polytope with $n(P) = 3^m -1$.
Then we have the following.
\begin{enumerate}
\item[\ref{thm:equality}.1] $P$ is a zonotope.
\item[\ref{thm:equality}.2] If $m=3$, then the angles of any face of $P$ are not acute, and $P$ has a rectangle-shaped face.
\end{enumerate}
\end{thm}

\begin{proof}
Consider some polytope $P \in \M_m$ satisfying $n(P) = 3^m -1$. Let $\F$ denote the family of faces of $P$.
For any face $F \in \F$, let $U_F$ denote the intersection of $P$ with the union of all the normals of $P$ through a point of $F$.
We note that if $\dim F = i$, then this set is called in \cite{Hann93} the \emph{$(m-i)$-wedge} of $P$, corresponding to $F$.
Let $U_F^s$ be the reflection of $U_F$ about the affine hull $\aff F$ of $F$.
The sets $U_F^s$ are mutually nonoverlapping for any two faces of $P$.

By \cite{Hann93}, we have $P \cup \bigcup_{F \in \F} U_F^s \subseteq 2P-P = P + (P-P)$.
Thus, for any $P \in \M_n$, $n(P) = 3^m -1 $ is equivalent to
\begin{equation}\label{eq:zonotopes}
P \cup \bigcup_{F \in \F} U_F^s = 2P-P = P + (P-P).
\end{equation}
This equality implies that the intersection of every normal of $P$, with $P$, is an affine diameter of $P$.

Let $F$ be any face of $P$, and $H_F$ be the \emph{linear} subspace parallel to $\aff F$.
Then the reflection of $F$ about $H_F$ is also a face of $P$.
Combining this property with the fact that $P$ is $o$-symmetric, we obtain that every face of $P$ is centrally symmetric, which yields
that $P$ is a zonotope (cf. \cite{McMullen}).

Now let $m=3$, and consider any edge $E=[x,y]$ of $P$.
Let $F_1$ and $F_2$ be the two faces of $P$ containing $[x,y]$.
Let $H$ be a plane supporting $P$ at exactly $[x,y]$.
Then $[-x,-y]$ is also an edge of $P$, and any segment $[u,v]$, with $u \in [x,y]$ and parallel to $[x,-y]$,
is an affine diameter of $P$.
Slightly rotating these segments about the line containing $E$, they remain inside $U_E$, and thus, these rotations are still affine diameters of $P$.
Hence, all these segments intersect $-F_1$ or $-F_2$.
From this, the first half of the assertion follows.

Finally, if $P$ has $f$ faces and $e$ edges, then, since the edge graph of its dual is planar, we have $e \leq 3f-6$.
From this, it follows that some face of $P$ contains strictly less than $6$ vertices.
More specifically, since each face of $P$ is centrally symmetric, it follows that $P$ has at least two pairs of parallelogram faces.
On the other hand, we have seen that no two consecutive edges of $P$ meet at an acute angle, which yields that any parallelogram face of $P$ is a rectangle.
\end{proof}

We note that the first part of (\ref{thm:equality}.2) holds also for $m > 3$, for the angles between two $(m-2)$-dimensional faces of $P$, contained in the same facet.

\subsection{Bodies containing all their centres of curvature}\label{subsec:curvature}

First, we generalize the methods of Chakerian \cite{C84} and Hann \cite{Hann93} for such plane convex bodies, for any dimension.

\begin{thm}\label{thm:constantwidth}
If $K$ is a convex body with $C^2$-class boundary and containing all of its centres of curvature, then $I(K) \leq \vol(K-K)$, and
$n(K) \leq \binom{2m}{m}$.
\end{thm}

For the proof, we need the following two lemmas.

\begin{lem}\label{lem:generalCSB}
If $f_1,f_2, \ldots, f_k$ are nonnegative integrable functions on $[a,b]$, then
\[
\int_a^b \prod_{i=1}^k f_i \leq \prod_{i=1}^k \sqrt[k]{\int_a^b f_i^k }.
\]
\end{lem}

\begin{proof}
We prove by induction.
For $k=1$, the assertion trivially follows.
Assume that it is true for some $k$, and consider nonnegative, integrable functions $f_1, \ldots, f_{k+1}$.
Then, using H\"older's inequality with $p=k+1$ and $q=\frac{k+1}{k}$ and applying the inductive hypothesis, we have
\[
\int_a^b \prod_{i=1}^{k+1} f_i \leq \sqrt[k+1]{\left(\int_a^b \prod_{i=1}^k f_i^{\frac{k+1}{k}} \right)^k} \sqrt[k+1] \int_a^b f_{k+1}^{k+1} \leq
\left( \prod_{i=1}^k \sqrt[k+1]{ \int_a^b f_i^{k+1} } \right) \sqrt[k+1]{\int_a^b f_{k+1}^{k+1} }.
\]
\end{proof}

\begin{lem}\label{lem:inequality}
Let $p(x)= \prod_{i=1}^k (x-x_i)$, where for every $i$, we have $0 \leq x_i \leq a$ for some given $a \in \Re$.
Then the maximum of $\int_0^a |p(x)| \dif x$ under these conditions is
\[
\int_0^a x^k \dif x = \frac{a^{k+1}}{k+1}.
\]
\end{lem}

\begin{proof}
By Lemma~\ref{lem:generalCSB}, we have
\[
\int_0^a |p(x)| \dif x \leq \prod_{i=1}^k \sqrt[k]{ \int_0^a |x-x_i|^k \dif x }
\]
and it suffices to prove that for every $x_0 \in [0,a]$ and positive integer $k$, we have
\[
\int_0^a |x-x_0|^k \dif x \leq \frac{a^{k+1}}{k+1}.
\]

Now, observe that
\[
\int_0^a |x-x_0|^k \dif x = \frac{x_0^{k+1}+(a-x_0)^{k+1}}{k+1} \leq \frac{a^{k+1}}{k+1},
\]
which yields the assertion.
\end{proof}


\begin{proof}[Proof of Theorem~\ref{thm:constantwidth}]
For convenience, we assume that $K$ has nowhere vanishing curvature, and remark that this method works also in the general case.

Let $p \in \bd K$, and parametrize a neighborhood of $p \in \bd K$ with local coordinates $(u_1,u_2,\ldots,u_{m-1})$.
Let $N(u_1,\ldots,u_{m-1})$ denote the inner unit normal vector of $\bd K$ at $p$.
Let $L(p)$ or equivalently, $L(u_1,u_2,\ldots,u_{m-1})$ denote the length of the part of the normal of $K$ at $p$, contained in $K$.
Consider the mapping
\begin{equation}\label{eq:mapping}
r(u_1,\ldots,u_{n-1},\lambda) = r(u_1,\ldots,u_{m-1}) + \lambda N(u_1,\ldots,u_{m-1}),
\end{equation}
where $0 \leq \lambda \leq L(u_1,u_2,\ldots,u_{m-1})$. Let $R$ denote the domain of this mapping.
By a theorem of Federer \cite[p.243]{F69}, we have
\begin{equation}\label{eq:Jacobi}
\int_K n(K,p) \dif p = \int_{R} |J(u_1,\ldots,u_{m-1},\lambda)| \dif \lambda \dif u_1 \ldots \dif u_{m-1},
\end{equation}
where $J$ is the Jacobian of the map in (\ref{eq:mapping}).

Note that $J$ is an $(m-1)$-degree polynomial of $\lambda$ with the determinant
$| \partial_{u_1} N, \ldots, \partial_{u_{m-1}} N, N|$ as its main coefficient.
On the other hand, it is known (cf. \cite{Hann93} or \cite{Spivak}) that this polynomial has only real roots, which are equal to the principal radii of curvature of $\bd K$ at the given point.
Thus, denoting these radii by $\rho_i$, $i=1,2,\ldots,n-1$, we can rewrite the integral in the form
\begin{equation}\label{eq:atteres}
\int_{R} |J(u_1,\ldots,u_{m-1},\lambda)| \dif u_1 \ldots \dif u_{m-1} = \int_{\Sph^{m-1}} \int_0^{L(u)} \prod_{i=1}^{m-1}|\lambda - \rho_i| \dif \lambda \dif u,
\end{equation}
where
\[
\dif u = | \partial_{u_1} N, \ldots, \partial_{u_{m-1}} N, N| \dif u_1 \ldots \dif u_{m-1}
\]
is the surface area element of the sphere $\Sph^{m-1}$, and $L(u)$ is the length of the part of the normal with tangent vector $u$, contained in $K$.

Let $D(u)$ be the length of the longest chord of $K$ in the direction of $u \in \Sph^{m-1}$.
Under our assumptions, for any $u \in \Sph^{m-1}$ and $i=1,2,\ldots,m-1$ we have that $0 \leq \rho_i \leq L(u) \leq D(u)$.
Thus, by Lemma~\ref{lem:inequality}, we have,
\[
I(K) = \int_{\Sph^{m-1}} \int_0^{D(u)} \prod_{i=1}^{m-1}|\lambda - \rho_i| \dif \lambda \dif u \leq \frac{1}{m} \int_{\Sph^{m-1}} D^m(u) \dif u = \vol(K-K).
\]
The second estimate follows from the inequality in \cite{RS57} for the volume of the difference body of $K$.
\end{proof}

\begin{thm}\label{thm:eikonal}
Let $K \subset \Re^m$ be a convex body, containing all its centres of curvature, with a $C^2$-class boundary.
Consider the deformation $C: \Sph^{m-1} \times [0,T) \mapsto $, of the embedding $C(\Sph^{m-1},0)=\bd K$, satisfying
$\frac{\partial C}{\partial t} = N$, where $N$ is the outer unit normal of $C$.
For any $t \in [0,T)$, let $K(t)$ denote the convex body bounded by $C(\Sph^{m-1},t)$. 
If $K$ is not a Euclidean ball, then $n(K(t))$ is a strictly increasing function of $t$.
If $K$ is a Euclidean ball, $n(K(t)) = 2$ for any value of $t$.
\end{thm}

\begin{proof}
Since $K$ contains all its centres of curvature, we have that $n(K,p) = 2$ for any $p \in \bd K$.
On the other hand, it is known \cite{Spivak} that for the above eikonal equation the centres of curvature of $K(t)$ 
do not change. Thus, for any $t \geq 0$, $K(t)$ contains all its centres of curvature, which yields that $n(K(t),p) = 2$ holds
for any $p \in K(t) \setminus K$.

Clearly, if $K$ is not a ball, $n(K) > 2$, from which it follows that $n(K(t)) < n(K)$ for any $t > 0$.
By the same argument, we obtain that for any $t_2 > t_1 \geq 0$, we have $n(K(t_2)) < n(K(t_1))$.
If $K$ is a Euclidean ball, then $K(t)$ is also a Euclidean ball for any $t \geq 0$, which readily implies the assertion in the second case.
\end{proof}

\subsection{Planar results}

Let us remark that, as the results of Hann \cite{Hann93} and Hug \cite{Hug95} show, the Lebesgue measure of the points $p \in \inter K$ with $n(K,p)=\infty$ is
zero. Nevertheless, our next example shows that, even in the smooth case, the condition $n(K,p)=\infty$ cannot be replaced with $n(K,p) \geq L$ for any given $L \in \Re$, depending only on $K$.

\begin{thm}\label{thm:example}
There is a plane convex body $K$, with $C^3$-class boundary, such that for any $L \in \Re$, the set of points $p \in \inter K$, satisfying the condition
$n(K,p) \geq L$, is of positive measure.
\end{thm}

\begin{proof}
Consider the semi-circle $y=\sqrt{1-x^2}$ in a Cartesian coordinate system. Let $\alpha_k = \frac{\pi}{2^k}$, where $k=1,2,\ldots$, and
$p_k= (\sin\alpha_k,\cos\alpha_k)$. Let $p=(0,1)$. For every value of $k$, connect the points $p_k$ and $p_{k+1}$ with a circle arc of radius $r_k=1+\sin^2 \alpha_{k+1}$, and let $f$ denote the function defined in this way.
Then the centre of the $k$th circle is $c_k=(u_k,v_k) = (-\sin^3 \alpha_k, -\sin^2 \alpha_k \cos \alpha_k)$.
Observe that $r_k >1$ for every $k$, and that the limit of the sequence $\{ r_k \}$ is $1$.
Clearly, $f$ is $C^\infty$ at every point, apart from the endpoints of the circle arcs, and possibly $x=0$.

It can be shown that $f$ is three times continuously differentiable at $x=0$.
Indeed, the fact that it is continuously differentiable follows from the geometric meaning of derivative.
To show that it is twice continuously differentiable,
we can use the definition of derivative. Finally, one can check that the curvature of $f$ is continuously differentiable at $x=0$, which yields the required statement.

As a last step, we can use the method described in \cite{Ghomi} to smoothen $f$ at the endpoints of the circle arcs, while preserving concavity of $f$ and its differentiability properties at $x=0$ in such a way that the curve $F$ obtained in this way contains at least two third of each circle arc.
Let us extend $F$ to be the boundary of a plane convex body $K$ with $C^3$-class boundary.
Then for every positive integer $L$, the origin $o$ has a neighborhood within which every point belongs to at least $L$ normals, each starting at a circle arc in $\bd K$.
\end{proof}

Our next result can be interpreted as a second indication, after the result in \cite{D98}, that ``in general'', $n(K)$ is greater for polytopes than for smooth convex bodies. Here, Theorem~\ref{thm:example} suggests that the conditions in Theorem~\ref{thm:discretization} are not only technical.

\begin{thm}\label{thm:discretization}
Let $K$ be a plane convex body, with $C^3$-differentiable boundary. Parametrize $\bd K$ as the curve $s \mapsto r(s)$, using arc-length parametrization,
where $0 \leq s \leq l=\perim K$. Assume that there is some $L \in \Re$ such that the measure of the points $p \in \inter K$ with $n(K,p) \geq L$ is zero.
For any integer $k \geq 3$, let $P_k = \conv \left\{ r\left( \frac{j}{k} l \right) : j=0,1,\ldots,k-1 \right\}$.
Then there is some $k_0 \in \Re$ such that for every $k \geq k_0$, we have $n(P_k) > n(K)$.
\end{thm}

\begin{proof}
For any $0 \leq s \leq l$, let $L_s$ and $c(s)$ denote, respectively, the normal line and the centre of curvature of $\bd K$ at $r(s)$.
Let $p \in L_s \cap \inter K$. It is an elementary computation to show that the second derivative of the Euclidean distance function $s \mapsto |r(s)-p|$
is positive, zero or negative if, and only if $p \in \relint [c(s),r(s)]$, $p=c(s)$ or $p \notin [c(s),r(s)]$, respectively.
Note that in the first case this function has a local minimum, in the third one a local maximum, and in the second $p$ belongs to the \emph{evolute} $E$
(or in other words, \emph{caustic}) of $\bd K$.
Following the terminology of dynamical systems, we say that in the first case $K$ has a \emph{stable equilibrium point} at $r(s)$, with respect to $p$,
and in the third one that it has an \emph{unstable equilibrium}.

Clearly, for any point $p \notin E$, the distance function has no degenerate critical point on $\bd K$, and thus, the number of stable and unstable
points, with respect to $p$, are equal to $\frac{n(K,p)}{2}$, which we denote by $u(K,p)$.
On the other hand, we can see from (\ref{eq:atteres}) that $E$ has zero measure.
Thus, we have $I(K) = 2 \int_{p \in K \setminus E} u(K,p) \dif p$, and it suffices to prove the assertion for the average number of stable points.

Let $\bd K$ have a stable equilibrium at $r(s_0)$ with respect to some $p \in \inter K$.
Then $p \in \relint [c(s_0),r(s_0)]$, and there is some $\varepsilon > 0$ such that on $[s_0-\varepsilon,s_0]$, the distance function of $\bd K$,
measured from $p$, is strictly decreasing, and on $[s_0,s_0 + \varepsilon]$ it is strictly increasing.
Examining the sign of the inner product $\langle r(s) - q, \dot{r}(s) \rangle$, it can be seen that with respect to any point $q \in \relint [p,r(s)]$,
the distance function has the same property, with the \emph{same} value of $\varepsilon$.

For $s \in [0,l]$, let $\kappa(s)\leq 0$ denote the curvature of $\bd K$ at $r(s)$, and observe that $|c(s)-r(s)| \kappa(s) = 1$.
Set
\[
A = \{ p \in \inter K: p \in (c(s),r(s)] \hbox{ and } 0.99 \leq |p-r(s)| \kappa(s) \leq 1 \hbox{ for some } s \in [0,l] \} \]
and
\[
B = (\inter K) \setminus A .
\]

First, we show that there is some $k_0$ such that for any $k \geq k_0$, and, for almost all $p \in B$,
$u(K,p) \leq u(P_k,p)$.
For any $s \in [0,l]$, let $\varepsilon(s)>0$ be the largest value such that for the point $p \in \relint [c(s),r(s)]$ with $|p-r(s)| \kappa(s) = 0.99$,
the function $t \mapsto |r(t) - p|$ is strictly decreasing on $[s-\varepsilon,s]$ and strictly increasing on $[s,s+\varepsilon]$.
Note that, by the argument in the third paragraph of the proof, the same property holds for any point $p \in \relint [c(s),r(s)]$ with
$|p-r(s)| \kappa(s) \leq 0.99$
Since $\varepsilon(s)$ depends continuously on $s$, there is some universal value $\varepsilon > 0$ such that
for any $s \in [0,l]$ and $p \in [c(s),r(s)]$ with $|p-r(s)| \kappa(s) \leq 0.99$, $|p-r(t)|$ is strictly decreasing on $[s-\varepsilon,s]$,
and strictly increasing on $[s,s+ \varepsilon]$.
Thus, as $\bd K$ is arc-length parametrized, for any $k \geq k_1=\frac{l}{\varepsilon}$ and point $p \in B$, there is a side of $P_k$ in the $\varepsilon$-neighborhood of any stable point of $K$ with respect to $p$.
Observe that if $p \in [c(s),r(s)]$ and $r(s)$ is not a vertex of $P_k$, then this side of $P_k$ contains a stable point with respect to $p$.
From this, it readily follows that, apart from the points of the normals of $K$ at the vertices of $P_k$, $u(K,p) \leq u(P_k,p)$ for any
$p \in B$.

Now, let $p \in A$. Then, for some $s \in [0,l]$, we have $p \in \relint [c(s),r(s)]$ and $0 \leq 1-|p-r(s)| \kappa(s) \leq 0.01$.
Using the idea of the proof of Theorem 1 in \cite{DLSZ12}, we obtain that there is a (universal) value $k_2$ such that
for any $k \geq k_2$ and any such point $p \in B$, we have $u(K,p) + 50 \leq u(P_k,p)$.

Finally, let $T_k = \area(K \setminus P_k)$.
By Blachke's Rolling Ball Theorem, $K \cap E \neq \emptyset$, and thus, $B$ is not empty.
Let $\area(B) = 2T>0$. We may assume that $\area(B \cap P_k) \geq T$, since it holds for sufficiently large values of $k$.
Then, for any $k \geq k_2$, we have that
\[
\int_{P_k} s(P_k,p) \dif p \geq \int_{A \cap P_k} u(K,p) \dif p + \int_{B \cap P_k} u(K,p)+50 \dif p =
\]
\[
= \int_{P_k} u(K,p) \dif p + 50 \area(B \cap P_k) \geq \int_K u(K,p) \dif A - L T_k + 50 T.
\]
Since $\lim_{k \to \infty} T_k = 0$, we have that for sufficiently large values of $k$, $\int_{P_k} s(P_k,p) \dif p > \int_K u(K,p) \dif p$.
As $\area(P_k) < \area(K)$, the assertion follows.
\end{proof}

In the remaining part we deal with the normed version of the original problem.

\begin{thm}\label{thm:normedplanes}
Let $K$ be a $C^2$-class, strictly convex body of constant width in a normed plane with unit disk $M$.
Then we have
\[
n_M(K) \leq \frac{6}{3-2 \tau(M)},
\]
where $\tau(M)$ is the ratio of the area of a largest area affine regular hexagon inscribed in $M$, to $\area (M)$, and this estimate cannot be improved.
\end{thm}

\begin{proof}
Since $K$ is of constant width, its central symmetrical $\frac{1}{2}(K-K)$ is a homothetic copy of $M$, and
every Minkowski normal of $K$ is an affine diameter.
Thus, for every $p \in \inter K$ we have $2 n_M(K,p) = d(K,p)$, and we may assume, without loss of generality, that $M=\frac{1}{2}(K-K)$.

Using (10) from the paper of Chakerian, we obtain
\[
I_M(K) = 2 \int_K d(K,p) \dif A = \int_0^{2\pi} \int_0^{D(\theta)} |\lambda - \rho(\theta)| \dif \lambda \dif \theta,
\]
where $D(\theta)$ is the length of the diameter of $K$ with angle $\theta$ with a horizontal line,
and $\rho(\theta)$ is the distance of the instantaneous centre of rotation of this diameter
from the corresponding endpoint.
Note that $0 \leq \rho(\theta) \leq D(\theta)$, and $\rho(\theta) + \rho(\theta + \pi) = D(\theta)$.
Thus, we have
\[
I(K) \leq \frac{1}{2} \int_0^{2\pi} D^2(\theta) \dif \theta = \area(K-K) = 4 \area(M).
\]

Observe that equality is ``approached'' if at each point, the instantaneous centre of rotation is close to one of the endpoints of the corresponding longest chord.
Thus, $I(K)$ is ``almost'' equal to $4 \area(M)$ if $K$ is ``almost'' a Reuleaux polygon in the norm of $M$.
On the other hand, Chakerian \cite{C66} proved that in the normed plane with unit ball $M$, among plane convex bodies of constant width two,
the one with minimal area is a minimal area Reuleaux triangle in the norm.
Furthermore, he showed that the area of this triangle is $\area(K) = 2 \area(M) - \frac{4}{3}\area(H)$,
where $H$ is a largest area affine regular hexagon inscribed in $M$.
From this, the assertion readily follows.
\end{proof}

We note that if $M$ is a Euclidean disk, we have $\tau(M)=\frac{3\sqrt{3}}{2\pi}$ and $\frac{6}{3-2 \tau(M)} = \frac{2\pi}{\pi-\sqrt{3}}$,
implying the estimate (1.10) in \cite{C84}.
Furthermore, since $\tau(M) \leq 1$, with equality if, and only if $M$ is an affine regular hexagon,
our bound is better than the one in \cite{Hann97} for the case that all the centres of circular curvature of $K$ are contained in $K$.

\section{Remarks and questions}\label{sec:remarks}

Table \ref{tab2} shows the presently known best upper bounds on the value of $n(K)$.
\begin{center}
\begin{table} \label{tab2}
\begin{tabular}{|l||c|c|c|c|}
\hline
 & $\K_m$ & $\M_m$ & $\C_m$ & $\W_m$ \\
\hline
\hline
$m=2$ & $12$ & $8$ & $6$ & $\frac{2\pi}{\pi-\sqrt{3}}$ \\
\hline
$m=3$ & $62$ & $26$ & $20$ & $20$ \\
\hline
$m$ & $\left( \frac{3}{2} \right)^m \binom{2m}{m} - 1$ & $3^m -1$ & $\binom{2m}{m}$ & $\binom{2m}{m}$ \\
\hline
\end{tabular}
\caption{Upper bounds of $n(K).$ Rows correspond to dimensions ($m$), columns correspond to   convex bodies ($\K_m$), $o$-symmetric convex bodies ($\M_m$), convex bodies containing all their centres of curvature ($\C_m$), and convex bodies of constant width ($\W_m$). New results in last two rows, last two columns.}
\end{table}
\end{center}
We remark that our estimate $I(K) \leq \vol(K-K)$ for any $K \in \W_m$ is ``sharp'' for any Reuleaux-polytope.
Unfortunately, unlike in the plane, the minimal volume bodies of constant width are not known; in the case $m=3$,
an 80-year-old conjecture states that these are the so-called \emph{Meissner bodies} (cf. \cite{BF34} or \cite{KW11}),
which, unlike the optimal body in the planar case, are not Reuleaux polytopes.
Thus, following the argument in \cite{C84}, it is not possible to find the exact bounds for $n(K)$ in $\W_m$.

Note that no example is known of a plane convex body $K$ satisfying $n(K) > 8$.

\begin{ques}
Prove or disprove the existence of a plane convex body $K$, with $n(K) > 8$.
\end{ques}

\begin{rem}\label{rem:Euler}
Let $P \in \K_m$ be a polytope. For any $i$-face $F$, let $U_F$ denote the $(n-i)$-wedge of $P$ corresponding to $F$. Let $\F_i$ be the family of $i$-faces of $P$.
Hann \cite{Hann93} proved that, using this notation,
\[
1 + (-1)^{n-1} = \frac{1}{\vol(P)} \sum_{i=0}^{n-1} (-1)^i \sum_{F \in \F_i} \vol(U_F).
\]
Nevertheless, this identity is a special case of the Poincar\'e-Hopf Theorem.
\end{rem}

\begin{proof}
Let $K \in \K_m$ be a convex body with $C^2$-class boundary.
Let $p \in \inter K$ be a point of the normal through the point $q \in \bd K$. Then $K$ has an equilibrium at $q$ with respect to $p$.
If the Hessian of the function $x \mapsto |x-p|$ is not zero, we say that this equilibrium is \emph{nondegenerate}.
The number of the negative eigenvalues of the Hessian is called the \emph{index} of the equilibrium.
If $K$ has only nondegenerate equilibria with respect to $p$, let $n_i(K,p)$ denote the number of the $i$-index equilibria of $K$ with respect to $p$.
According to the Poincar\'e-Hopf Theorem, for any such point $p$, we have
\[
1 + (-1)^{m-1} = \sum_{i=0}^{m-1} (-1)^i n_i(K,p).
\]
It is well-known that $K$ has a degenerate equilibrium at some $q$ with respect to $p$ if, and only if $|q-p|$ is a principal radius of curvature at $q$
cf. e.g. \cite{PS96}).
Thus, by (\ref{eq:atteres}), the set of reference points with degenerate equilibria is of measure zero.
From this, by integration we obtain
\[
1+(-1)^{m-1} = \frac{1}{\vol(K)} \sum_{i=0}^{m-1} (-1)^i \int_{K} n_i(K,p) \dif p.
\]

To show the assertion for polytopes, we may use a standard smoothing technique, and observe that as $K \to P$,
$\int_{K} n_i(K,p) \dif p \to \sum_{F \in \F_i} \vol(U_F)$.
\end{proof}

\begin{rem}\label{rem:maximalpolyhedra}
The problem of finding the polytopes $P \in \M_m$ satisfying $n(P) = 3^m -1$ seems to be significantly harder for $m > 2$ (or even for $m=3$) than for $m=2$. One can easily check, for example, that if $P_0$ is a centrally symmetric polytope, inscribed in a circle, and $S$ is a closed segment perpendicular to the plane of $P_0$, then the prism $P=P_0 + S \subset \Re^3$ satisfies $n(P) = 26$.
On the other hand, among the five primary parallelohedra (which are all zonotopes), the cubes, the hexagonal prisms and the truncated octahedra  are maximizers in $\M_3$, whereas, according to (\ref{thm:equality}.2), elongated and rhombic dodecahedra are not.
\end{rem}

As it was observed in \cite{Hug95}, the functional $n(.) : \K_m \to \Re$ is not continuous, to show this one may 
consider a Euclidean disk $\B^2 \subset \Re^2$ with $n(\B)=2$, and an inscribed centrally symmetric $2k$-gon $P_k$ with $n(P_k)=8$.
Furthermore, using ``smoothened'' polygons in this example, one can show the discontinuity of this functional even in the family of convex bodies with $C^1$-class
boundary. On the other hand, for convex bodies with $C^2$-class boundaries, $n(.)$ is clearly continuous.
The situation is different if we consider local deformations only.
If we permit only truncations of a polytope or a $C^2$-class body by a plane, $n(K)$ changes continuously.
It would be interesting to know if the same holds for any convex body and for any local deformation.
In this direction, perhaps the methods of \cite{Hug95} can be useful.

To better understand abrasion processes it seems interesting to find bounds on the following quantity.

\begin{defn}\label{defn:surfaceaverage}
Let $K \in \K_m$. We set
\[
I_{surf}(K) = \int_{\bd K} n(K,p) \dif p, \quad \hbox{and} \quad n_{surf}(K) = \frac{I_{surf}(K)}{\surf(K)},
\]
where $\surf(K)$ denotes the surface area of $K$.
\end{defn}

\begin{prob}
Find the minima and maxima of $n_{surf}(K)$ over $\K_m$, $\M_m$,  $\C_m$ and $\W_m$, if they exist.
\end{prob}

\begin{rem}
Using the notations in the proof of Theorem~\ref{thm:eikonal}, it is easy to check that for the eikonal equation given there,
we have
\[
\left. \frac{\dif n(K(t))}{\dif t} \right|_{t=0} = \frac{\surf(K)}{\vol(K)} \left( n_{surf}(K) - n(K) \right).
\]
Hence, Theorem~\ref{thm:eikonal} can be interpreted as observing that, for any $K \in \K_m$ with $C^2$-class boundary and containing all its centres of curvature, we have $n_{surf(K)} \leq n(K)$.
\end{rem}

\begin{conj}
For any $K \in \K_m$ with $C^2$-class boundary, we have $n_{surf}(K) \leq n(K)$.
\end{conj}

We conjecture the following, more general version of Theorem~\ref{thm:eikonal}.

\begin{conj}
Let $K \subset \Re^m$ be a convex body, containing all its centres of curvature, with a $C^2$-class boundary.
Consider the deformation $C: \Sph^{m-1} \times [0,T) \mapsto $, of the embedding $C(\Sph^{m-1},0)=\bd K$, satisfying
$\frac{\partial C}{\partial t} = \frac{1}{\kappa^r} N$, where $N$ is the outer unit normal of $C$, $\kappa$ is its Gaussian curvature, and $r > 0$.
For any $t \in [0,T)$, let $K(t)$ denote the convex body bounded by $C(\Sph^{m-1},t)$. 
Then $n(K(t))$ is an increasing function of $t$.
\end{conj}

Instead of the average number of normals through a point of a given convex body, we may investigate the average $d(K)$ of the affine diameters
of $K$ (cf. Definition~\ref{defn:diameter}).

\begin{rem}
Hammer \cite{H51} proved that for any $K \in \K_m$ and any $p \in K$ there is an affine diameter of $K$ passing through $p$.
Thus, we have $d(K) \geq 1$. On the other hand, for the Euclidean unit ball $\B$ in $\Re^m$, we have $d(\B) = 1$. This determines the minimum of $d(K)$ over $\K_m$.
\end{rem}

To determine the maximum of $d(K)$ for sufficiently smooth and strictly convex bodies, one may try to follow the idea of Chakerian \cite{C84} and Hann \cite{Hann93}, replacing the unit normal vector $N$ of $\bd K$ in the proof of Theorem~\ref{thm:constantwidth} by the unit tangent vector of the affine diameter starting at the corresponding point of $\bd K$.
\emph{If} the Jacobian $J$ in (\ref{eq:Jacobi}) has only real roots, then we may apply the argument in the proof of Theorem~\ref{thm:constantwidth}
and obtain the estimate $D(K) \leq \frac{1}{2} \vol(K-K)$, generalizing (1.1) of \cite{C84} for any dimensions.
Nevertheless, this quantity is known to have only real roots only in the planar case, or if $N$ is orthogonal to $\bd K$ at each point.
This leads to the following questions.

\begin{ques}
Prove or disprove the existence of a strictly convex body $K \in \K_m$ with $C^2$-class boundary, for which the Jacobian of the mapping
\begin{equation}\label{eq:affinemapping}
r(u_1,\ldots,u_{m-1},\lambda) = r(u_1,\ldots,u_{m-1}) + \lambda N(u_1,\ldots,u_{m-1}),
\end{equation}
where $N$ is an inner unit tangent vector of the affine diameter starting at $r(u_1,\ldots,u_{m-1}) \in \bd K$, has nonreal roots.
\end{ques}

\begin{ques}
Prove or disprove the existence of a strictly convex body $K$ with $C^2$-class boundary, satisfying $D(K) > \frac{1}{2} \vol(K-K)$.
\end{ques}

On the other hand, affine diameters correspond to normals if, and only if the body is of constant width. Thus, our result immediately yields the following
generalization of the result of Chakerian \cite{C84}.

\begin{rem}
Let $K \subset \W_m$ be a convex body of constant width. Then $D(K) \leq \frac{1}{2}\vol(K-K)$.
\end{rem}

Note that a real root of the Jacobian of the mapping in (\ref{eq:affinemapping}) corresponds to a point $p$ on the affine diameter
where, in a certain direction, the family of lines, defined in the mapping, is rotated about $p$.
These points are the higher dimensional analogues of the \emph{instantaneous centre of rotations} used in \cite{C84}.

\end{document}